\documentclass[11pt]{amsart}

\usepackage{graphicx, amsmath, amsfonts, amsthm, amssymb, fullpage, caption, subcaption, microtype, verbatim, color}

\newtheorem{thm}{Theorem}[section]
\newtheorem{prop}[thm]{Proposition}

\newtheorem{lem}[thm]{Lemma}

\newcommand{\Z}{\mathbb{Z}}
\newcommand{\pic}[1]{\begin{minipage}{0.8in} \[\includegraphics[width=0.75in]{{#1}} \]\end{minipage}}

\title{Presentations of Roger and Yang's Kauffman Bracket Arc Algebra}

\author{Martin Bobb}
\address{Department of Mathematics, University of Texas at Austin, Austin, TX 78712, USA}

\author{Dylan Peifer}
\address{Department of Mathematics, Cornell University, Ithaca, NY 14853, USA}

\author{Stephen Kennedy}
\address{Department of Mathematics, Carleton College, Northfield, MN 55057, USA}

\author{Helen Wong}
\address{Department of Mathematics, Carleton College, Northfield, MN 55057, USA}

\thanks{The last author was supported in part by NSF Grant DMS-1105692. }

\begin{document}

\maketitle

\begin{abstract}
We provide a presentation of the Roger and Yang's Kauffman bracket arc algebra for the once-punctured torus and punctured spheres with three or fewer punctures.  
\\
\end{abstract}



Roger and Yang's Kauffman bracket arc algebra is a generalization of the well-known Kauffman bracket skein algebra of a surface, whose definition in \cite{Tu88, Pr91} is based on Kauffman's skein theoretic description of the Jones polynomial for knots and links (\cite{Jo85, Ka87}).   Later, the skein algebra  of a hyperbolic surface was interpreted as a quantization of the surface's Teichm\"uller space from hyperbolic geometry (\cite{TuraevPoisson, BullockFrohmanJKB, PrzSikora}).   Interest thus grew for the skein algebra, as a construction important in the Jones polynomial skein theory but also deeply related to Teichm\"uller theory.  

Following this body of work on the skein algebra, Roger and Yang introduced a ``skein algebra of arcs'' to be a skein theory version of Penner's decorated Teichm\"uller space.  In \cite{Penner}, Penner defined the decorated Teichm\"uller space as an alternate way to describe the hyperbolic structures of a surface using lengths of both simple closed curves and arcs between punctures on the surface (each decorated with a choice of horoball).  Roger and Yang defined their arc algebra as a quantization of Penner's decorated Teichm\"uller space, roughly in the same way that the skein algebra is a quantization of the usual Teichm\"uller space.   The arc algebra includes both simple closed curves and arcs between  punctures on the surface.  In addition to the two usual bracket skein relations for framed links, there are two extra relations for arcs and loops near the punctures. 

Understanding the algebraic structure of Roger and Yang's  arc algebra is an important first step to exploring its role as an intermediary between quantum topology and hyperbolic geometry.  Here, we seek  finite presentations of the  arc algebra for some simple surfaces, namely for spheres with three or fewer punctures and for tori with one or no punctures.   In the companion paper \cite{BobbPeiferWong}, we show that the  arc algebra is finitely generated.   Our work is inspired by analogous statements for the  skein algebra found in \cite{BuPr00} and \cite{Bu99}, respectively.

\section{The Kauffman Bracket Arc Algebra} \label{sec:defns}

Let $F_{g,n}$ denote a compact, orientable surface of genus $g$ with $n$ points $ p_1, p_2, \ldots, p_n$ removed.   The points removed are the \emph{punctures}.    Let $A$ be an indeterminate, with formal square roots $A^{\frac12}$ and $A^{-\frac12}$.   In addition, let there be an indeterminate $v_i$ associated to each puncture $p_i$, and let $R_n = \Z[A^{\pm\frac{1}{2}}][v_1^{\pm1}, v_2^{\pm1}, \dots, v_n^{\pm1}]$  denote the ring of Laurent polynomials in the commuting variables $A^{\frac12}$ and $v_1, \ldots, v_n$.      

A framed curve in the thickened surface $F_{g,n} \times [0,1]$ is the union of  framed knots and framed arcs that go from puncture to puncture.   (See \cite{RoYa14} for a precise definition.)    
Let $\mathcal G(F_{g,n})$ be the $R_n$-module freely generated by the framed curves in $F_{g,n} \times [0,1]$, up to isotopy, and let $\mathcal K(F_{g,n})$ be the submodule generated by terms of the following four forms:
\begin{align*}
&1)
\quad
\begin{minipage}{.5in}\includegraphics[width=\textwidth]{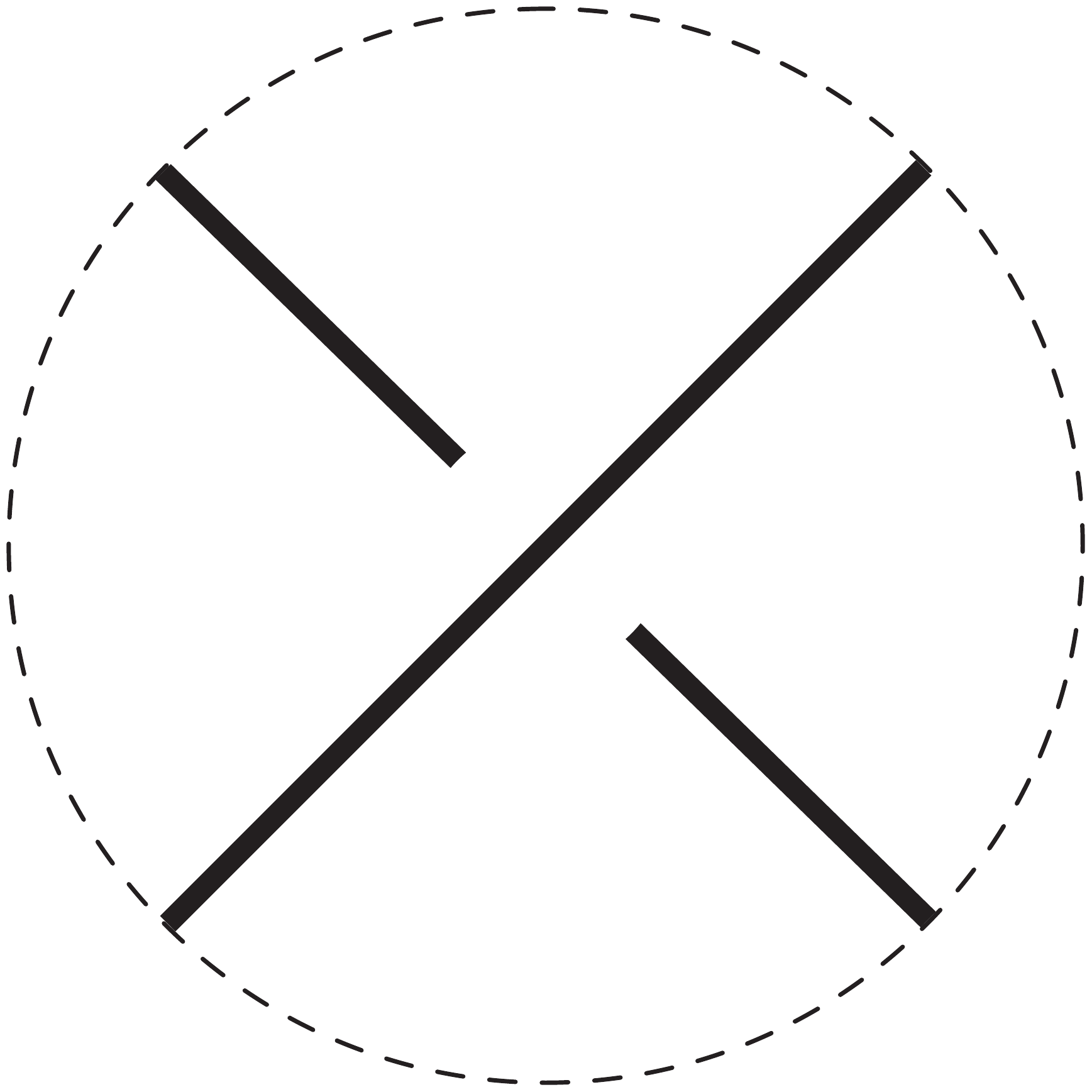}\end{minipage} 
-  \left( A\begin{minipage}{.5in}\includegraphics[width=\textwidth]{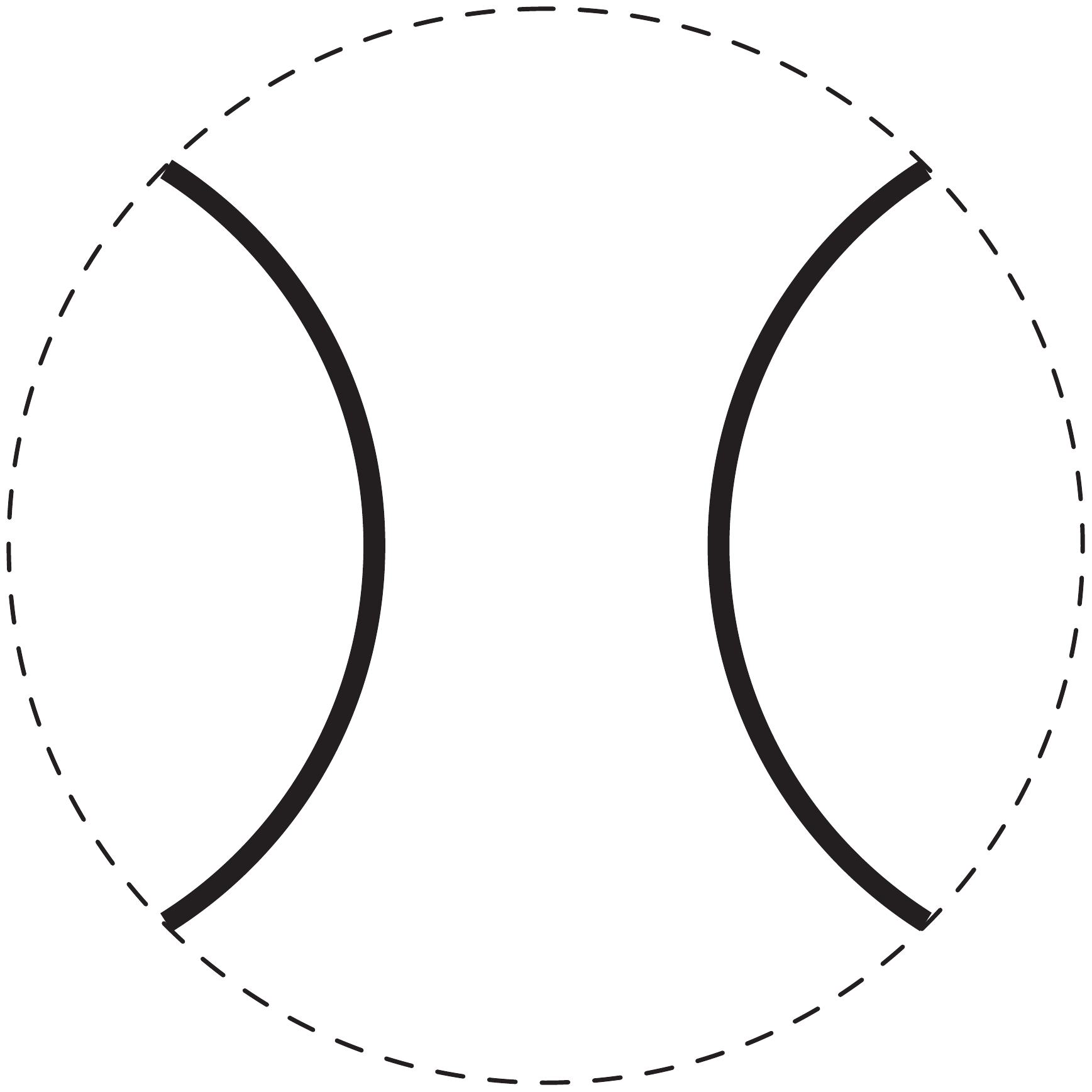}\end{minipage} 
+A^{-1}\begin{minipage}{.5in}\includegraphics[width=\textwidth]{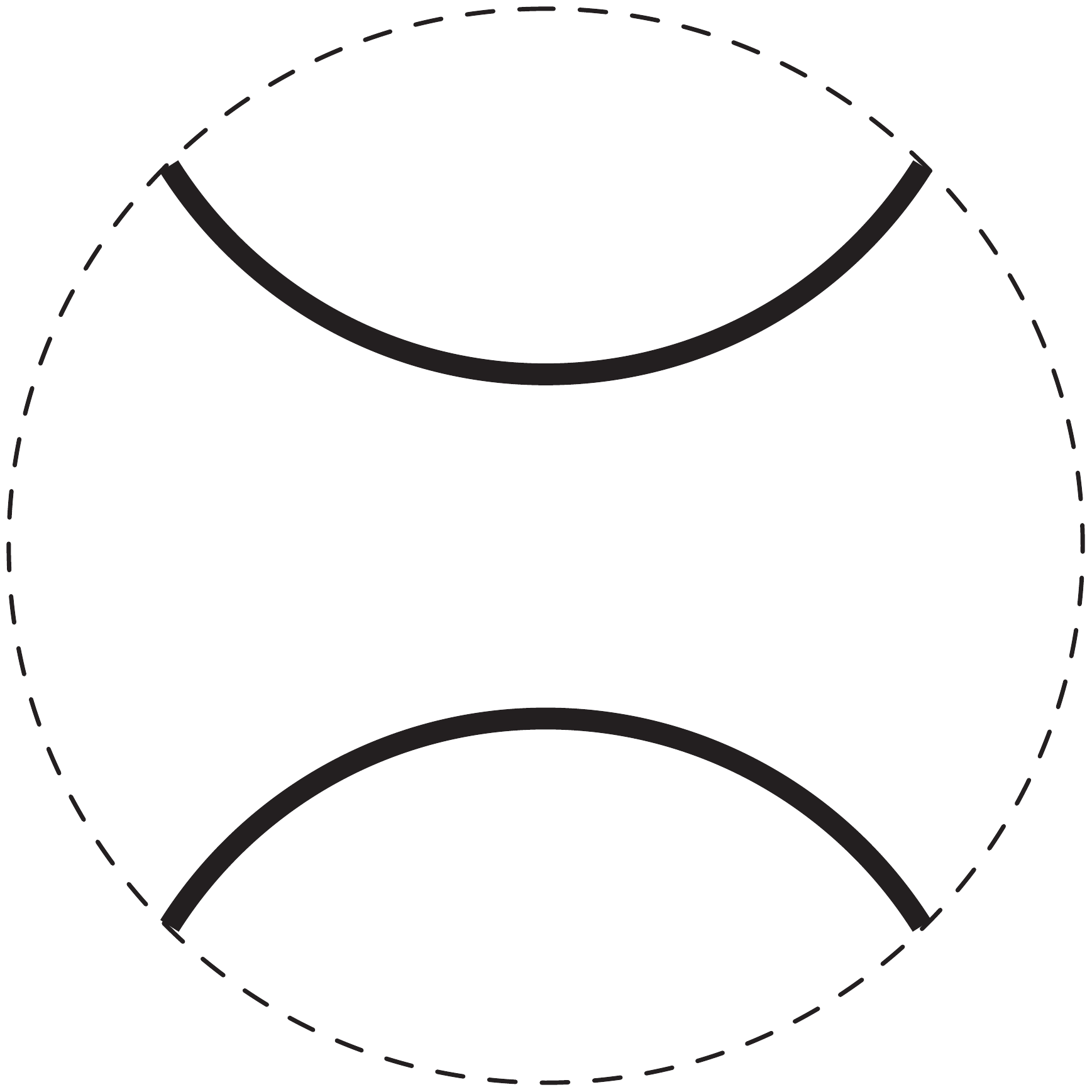}\end{minipage}  \right)
&\text{Skein Relation}\\
&2)
\quad 
v_i \begin{minipage}{.5in}\includegraphics[width=\textwidth]{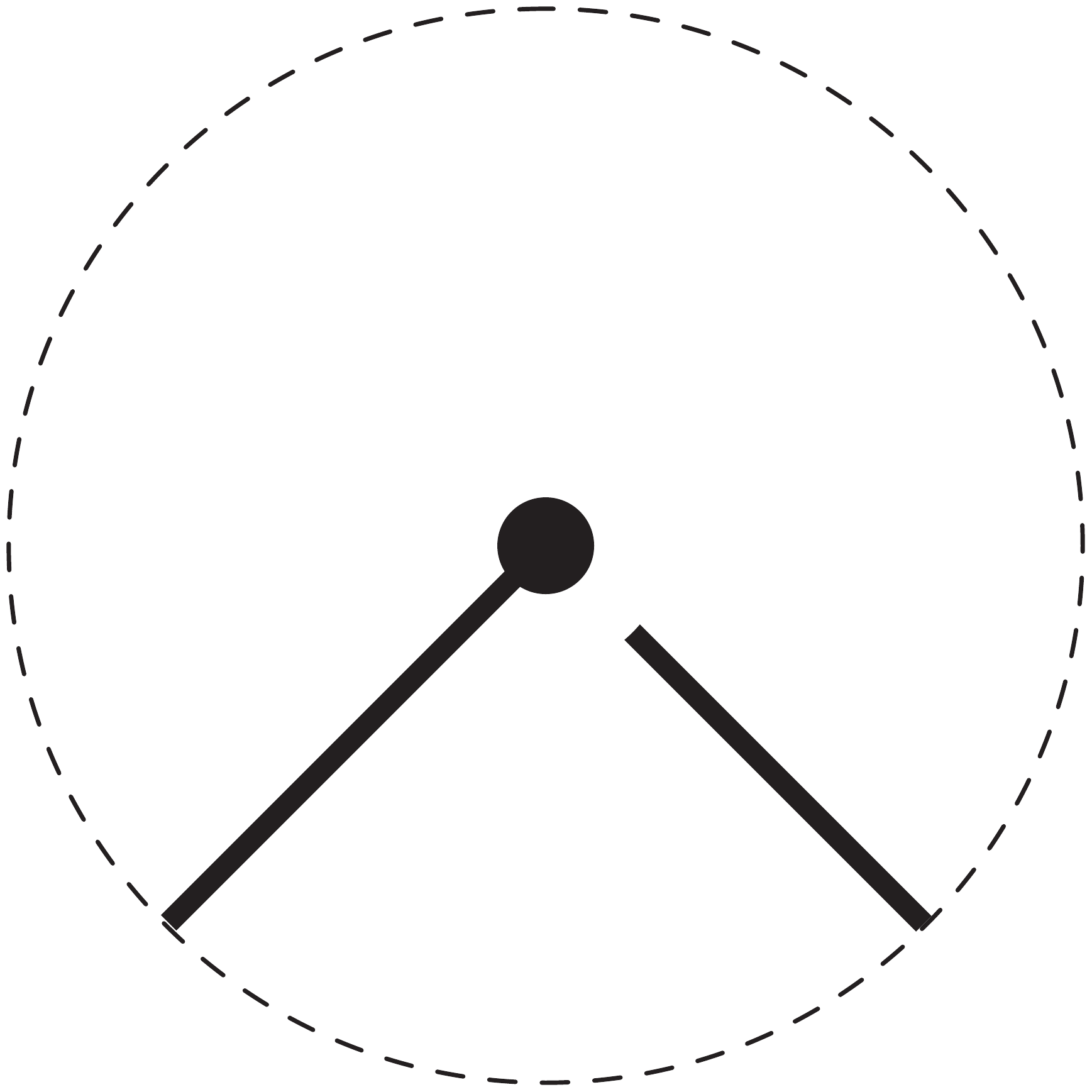}\end{minipage} 
- \left(  A^\frac{1}{2}\begin{minipage}{.5in}\includegraphics[width=\textwidth]{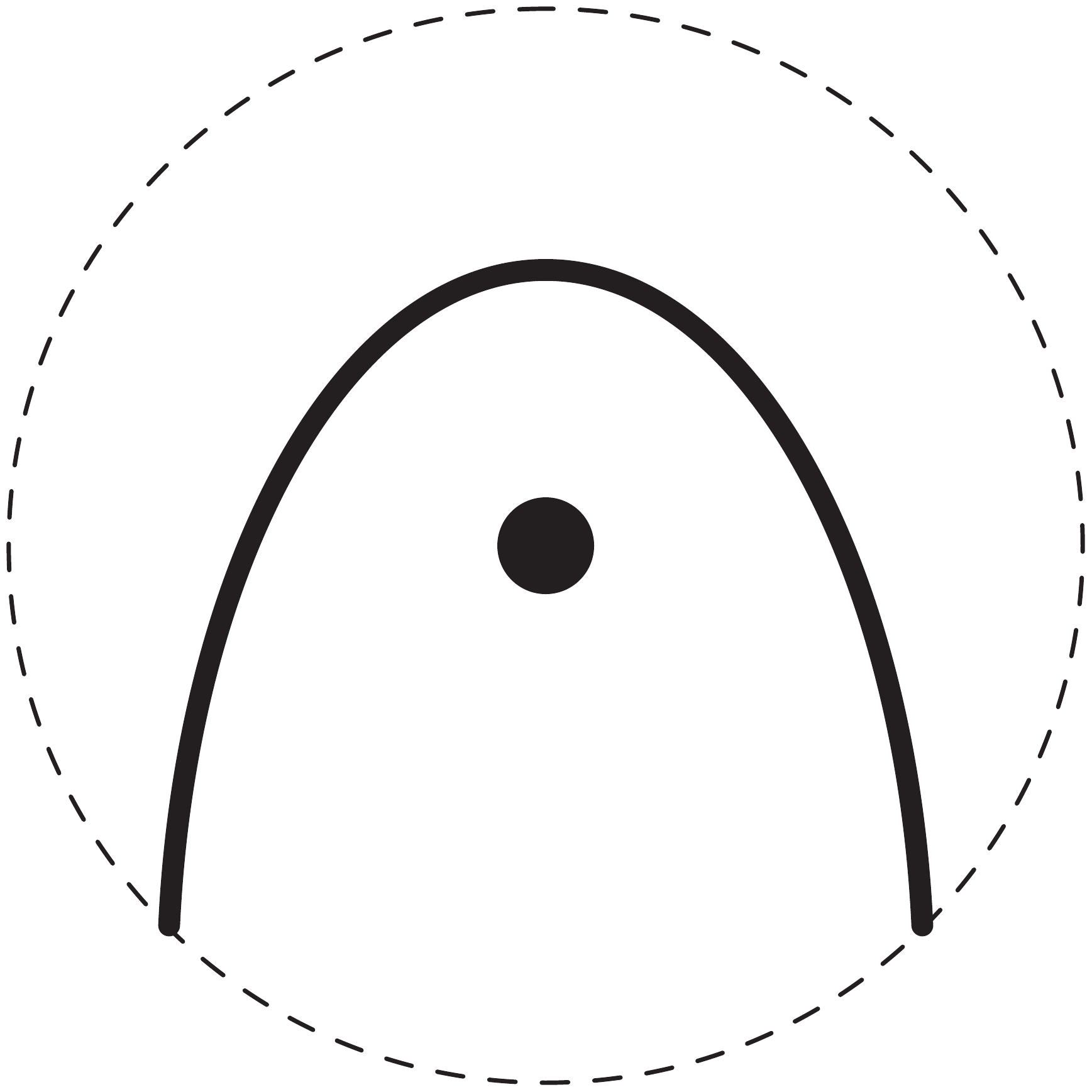}\end{minipage} 
+ A^{-\frac{1}{2}}\begin{minipage}{.5in}\includegraphics[width=\textwidth]{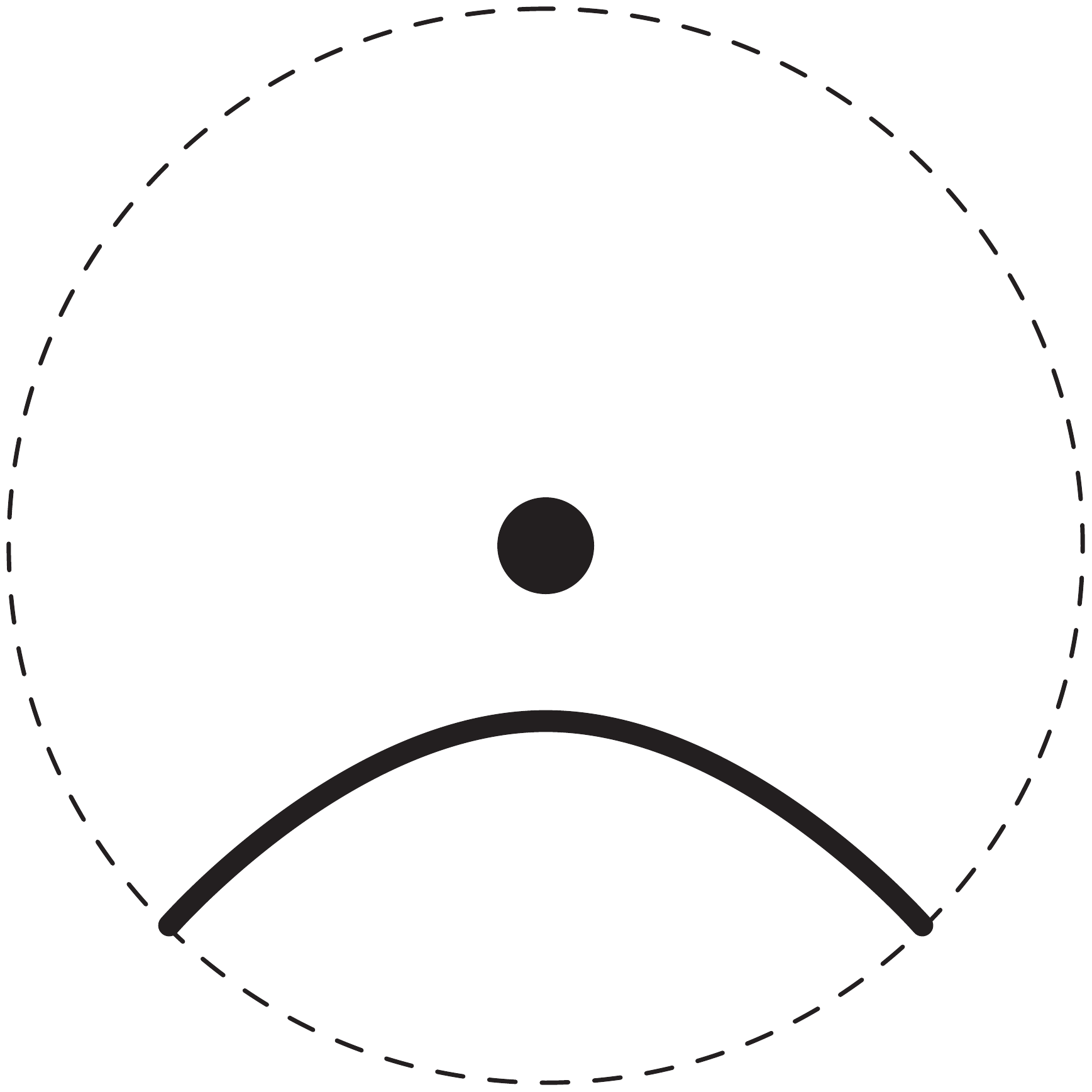}\end{minipage}  \right)
&\text{Puncture-Skein Relation on $i$th puncture} \\
&3)
\quad 
\begin{minipage}{.5in}\includegraphics[width=\textwidth]{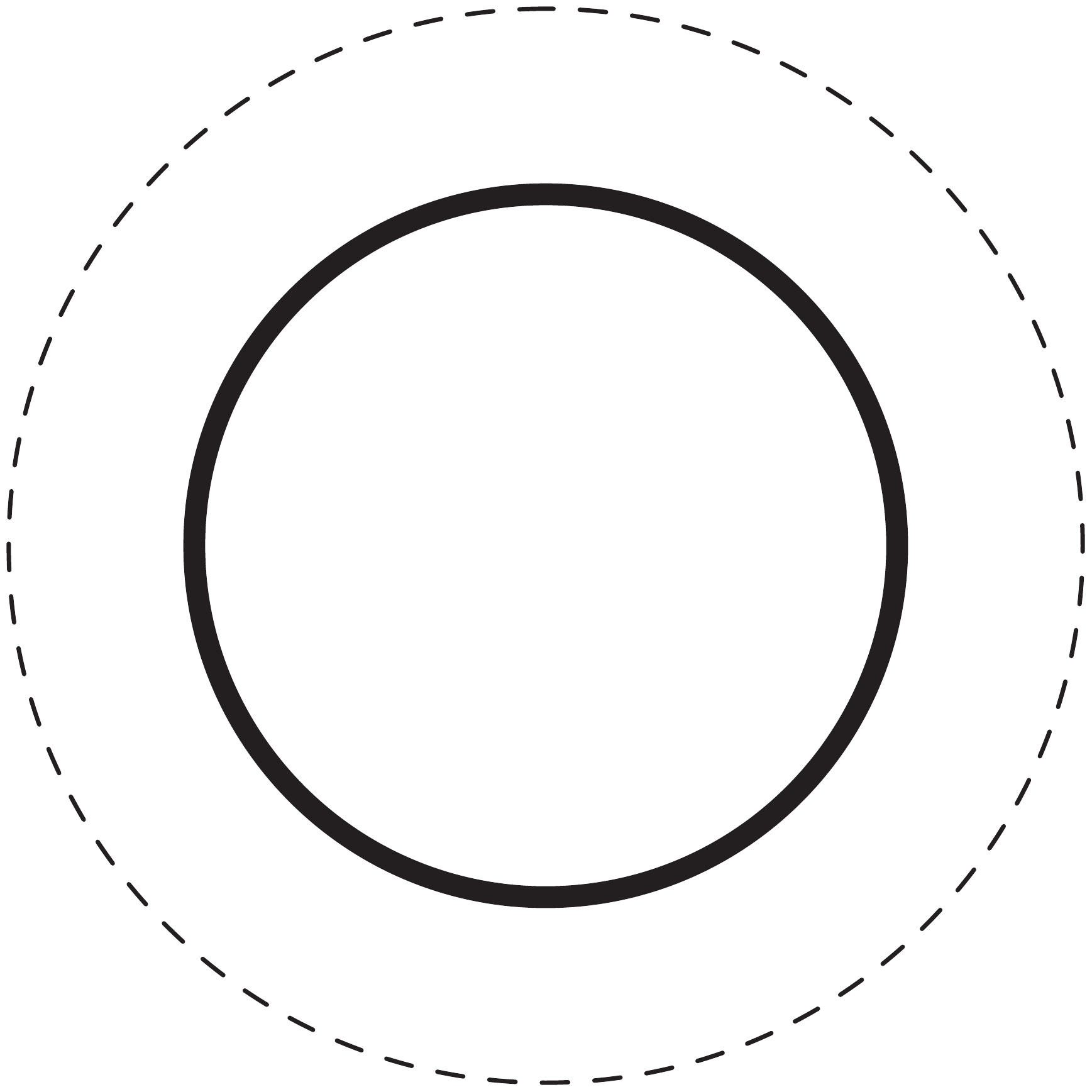} \end{minipage} 
- ( - A^2 - A^{-2} )
&\text{Framing Relation} \\
&4)
\quad 
\begin{minipage}{.5in}\includegraphics[width=\textwidth]{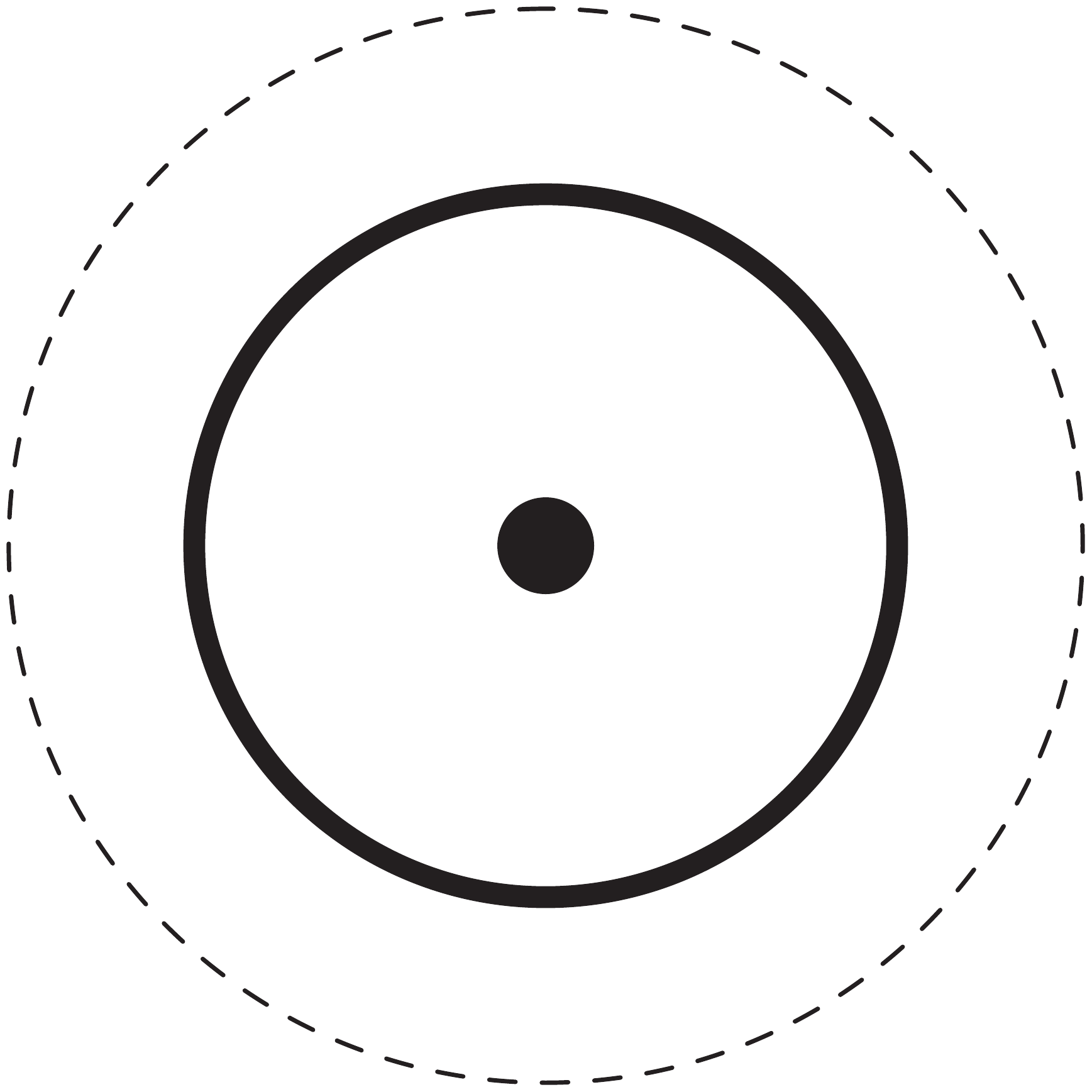} \end{minipage} 
-( A + A^{-1} )
&\text{Puncture-Framing Relation} 
\end{align*}
where the diagrams in each form are assumed to be identical outside of the small balls depicted. 
Let  $\mathcal A(F_{g,n})$ denote the quotient $\mathcal G(F_{g,n}) / \mathcal K(F_{g,n})$.

There is a natural stacking operation for framed curves in the thickened surface $F_{g,n} \times [0,1]$ which extends to $\mathcal A(F_{g,n})$.  That is, if $[L_1], [L_2] \in \mathcal A(F_{g,n})$ are respectively represented by framed curves $L_1, L_2$ in $F_{g,n} \times[0, 1]$, the product $[L_1]\ast[L_2] = [L_1' \cup L_2'] \in \mathcal A(F_{g,n})$ is represented by the union of the framed curve $L_1' \subset F_{g,n} \times [0, \frac12]$ (obtained by rescaling $L_1$ in $ F_{g,n} \times [0, 1]$) and of  the framed curve $L_2' \subset F_{g,n} \times [\frac12, 1]$ (obtained by rescaling $L_2$ in $F_{g,n} \times [0, 1]$).    This stacking operation makes $\mathcal A(F_{g,n})$ into an algebra, called the \emph{Kauffman bracket arc algebra} of the surface $F_{g,n}$. 

Diagrams in this paper represent projections of framed curves onto $F_{g,n}$ with over- and under-crossing information depicted by breaks in the projection at double-points in the projection, and where the framing of curves is vertical, at right angles to the plane of the paper.     Although more than one component of a framed curve can end at any puncture, they must do so at different heights.   Diagrams will indicate the order in height of the crossings as necessary.   
\begin{figure}
\includegraphics[width=.20\textwidth]{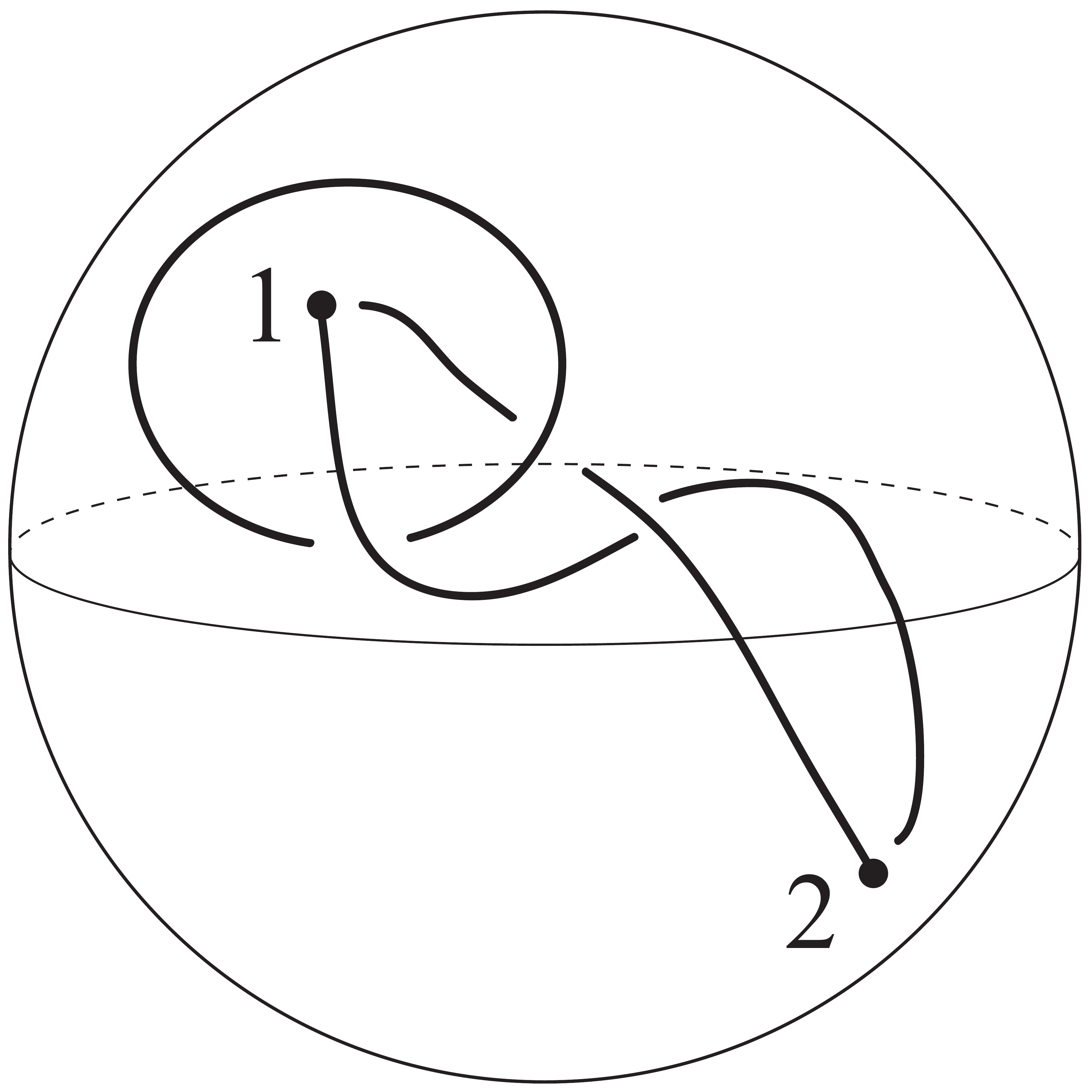}
\caption{A framed curve with three components on $F_{0, 2}$.}
\label{fig:arcex}
\end{figure}
Figure~\ref{fig:arcex} shows a framed curve consisting of three components (two framed arcs and a framed knot) in a sphere with two punctures.  No further labeling at punctures is necessary in Figure~\ref{fig:arcex} since arcs intersect each puncture only twice.  

Figure~\ref{fig:relnex} shows a product of two framed curves on a twice-punctured torus.  The product can be simplified by using a Reidemeister 2 move followed by Relation~2 (Puncture-Skein Relation)  to ``pull off" a pair of strands that meet at a puncture and  Relation~4 (Puncture Relation) to ``remove" trivial components enclosing a puncture.  
\begin{figure}[htpb]
\begin{eqnarray*}
\pic{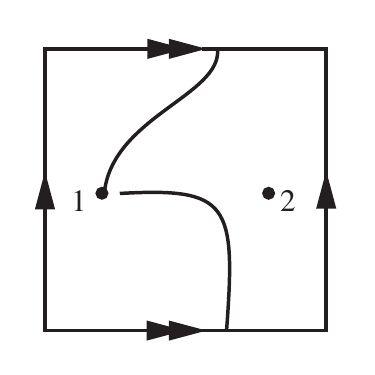} \!\! \!\! \ast \!\!\!\! \pic{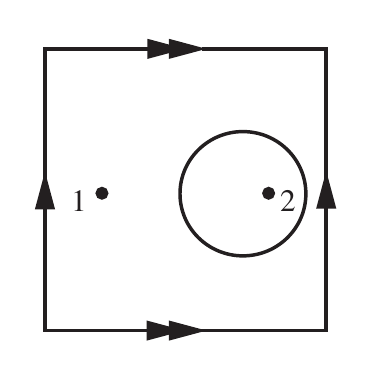} 
&=&\pic{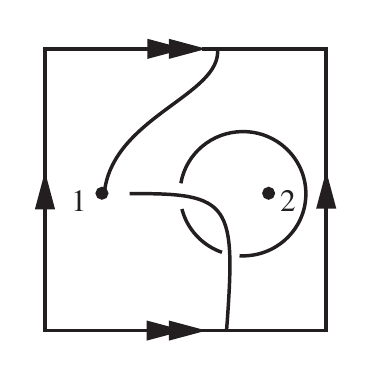} 
=\pic{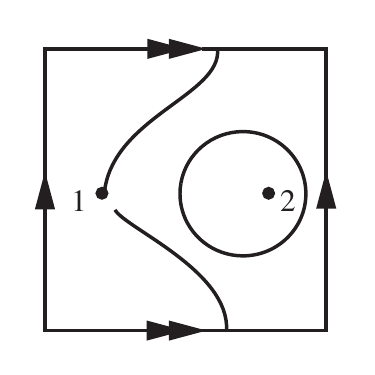} = \\
&=&  v_1^{-1} \left( A^\frac{1}{2}\pic{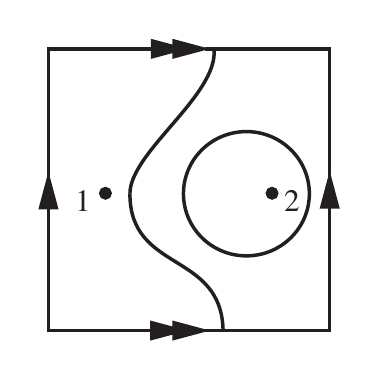} + A^{-\frac{1}{2}}\pic{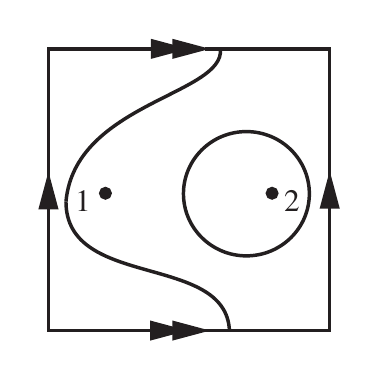} \right)\\
&=& v_1^{-1} (A+A^{-1}) \left( A^\frac{1}{2} \pic{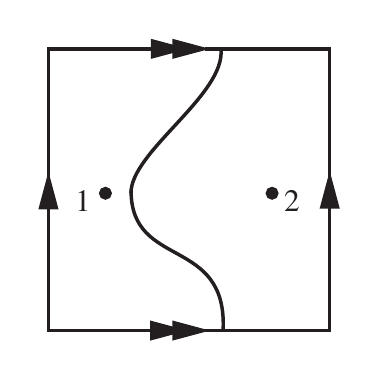} +  A^{-\frac{1}{2}} \pic{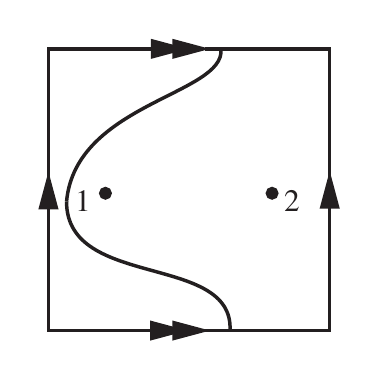} \right)
\end{eqnarray*}
\caption{Rewriting a framed curve in $\mathcal A(F_{1, 2})$.}
\label{fig:relnex}
\end{figure}

The Kauffman bracket skein algebra $\mathcal S(F_{g,n})$ defined by Turaev and Przytycki (\cite{Tu88, Pr91}) is closely related to the  arc algebra $\mathcal A(F_{g,n})$.  Recall that the  skein algebra $\mathcal S(F_{g,n})$ can be constructed by considering the quotient $\mathcal G_0(F_{g,n}) / \mathcal K_0(F_{g,n})$, where $\mathcal G_0(F_{g,n})$ is the $\Z[A, A^{-1}]$-module generated by the framed links in $F_{g,n} \times [0,1]$ and $\mathcal K_0(F_{g,n})$ is the submodule generated by only Relation 1 (the skein relation) and Relation 3 (the framing relation) from above.  Again, multiplication is induced by the stacking of framed links in the thickened surface.   Compared with the skein algebra $\mathcal S(F_{g,n})$, the definitions of the arc algebra $\mathcal A(F_{g,n})$ differs in two ways: in the choice of a larger ring and in the inclusion of two extra relations.    

\begin{lem} \label{lem:psi}
There exists a well-defined algebra homomorphism $\psi: \mathcal S(F_{g,n}) \to \mathcal A(F_{g,n})$.
\end{lem}

\begin{proof}
Consider the map $i: \mathcal G_0(F_{g,n}) \to R_n \otimes \mathcal G_0(F_{g,n})$ with $i(x) = 1 \otimes x$ that changes the scalars, and the map $j:  R_n \otimes \mathcal G_0(F_{g,n})  \to \mathcal G(F_{g,n})$ with $j(p \otimes x)=  p \cdot x$ that includes the framed links into the framed curves.   Let $\hat \psi = j \circ i$.  Notice that $\hat \psi (\mathcal K_0(F_{g,n}))  \subseteq  \mathcal K(F_{g,n})$.  Thus $\hat \psi: \mathcal G_0(F_{g,n})  \to \mathcal G(F_{g,n}) $ descends to a map $\psi:  \mathcal G_0(F_{g,n})/ \mathcal K_0(F_{g,n}) \to \mathcal G(F_{g,n})/\mathcal K (F_{g,n})$. 
\end{proof}

We are interested in the image of $\psi$.  In certain small cases, it generates $\mathcal A(F_{g,n})$, a fact we will exploit later in Section \ref{sec:n01}.

\section{Generators and Relations for the Arc Algebra}

A general strategy for finding generating sets for $\mathcal A(F_{g,n})$ is to rewrite framed curves using ones with fewer crossings.  We say that a framed knot  is a \emph{simple knot} if it allows a projection onto $F_{g,n} \times \{0\}$ without any crossings and it does not bound a disk containing one or no punctures.    A framed arc  is a \emph{simple arc} if its endpoints are at distinct punctures and it allows a projection without any crossings.  A \emph{simple curve} is either a simple knot or a simple arc.   

\begin{lem} \label{lem:simple}
If a set of elements $\{x_1, x_2, \dots x_m\}$ generates the simple curves then it generates all of $\mathcal A(F_{g,n})$.
\end{lem}

\begin{proof}
Suppose we have a basis element $[L] \in \mathcal A(F_{g,n})$ represented by a framed curve $L\subset F_{g,n} \times[0, 1]$.
By application of the skein relation and the puncture-skein relation, $[L]$ may be written as a linear combination of skeins represented by framed curves each of which has no crossings and which intersects a puncture at most once.   In particular, the connected components of each framed curve can be isotoped to be at different heights, so $[L]$ is a linear combination of products of simple knots, simple arcs, and possibly some loops that bound disks containing one or no punctures.  Those latter loops may be removed by application of the framing and puncture relations.    Thus $[L]$ is a linear combination of simple curves.  Since $\{x_1, x_2, \dots x_m\}$ generate the simple curves,  $[L]$ is also in the set generated by $\{x_1, x_2, \dots x_m\}$.
\end{proof}

\subsection{Remark} Observe that if $(A^2 -1)$ is invertible, then the puncture-skein relation implies that:
\begin{align*}
\begin{minipage}{.5in}\includegraphics[width=\textwidth]{rel-punctureskein3.pdf}\end{minipage}  
&= v_i A^\frac{1}{2} (A^2 -1)^{-1}
 \left(  \begin{minipage}{.5in}\includegraphics[width=\textwidth]{rel-punctureskein1.pdf}\end{minipage} 
-  A \begin{minipage}{.5in}\includegraphics[width=\textwidth]{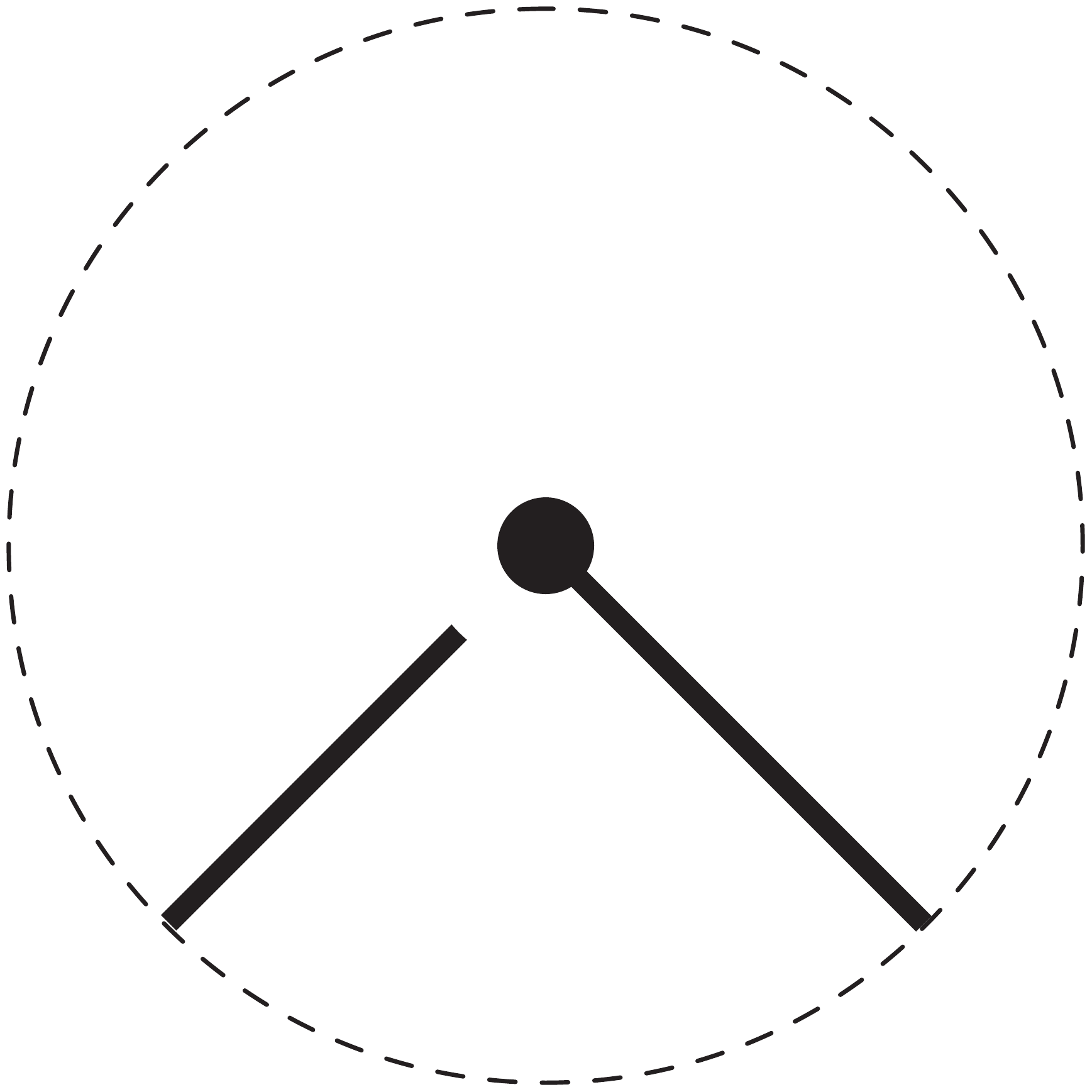}\end{minipage} 
 \right) 
  \\
 \mbox{ and  \; }
\begin{minipage}{.5in}\includegraphics[width=\textwidth]{rel-punctureskein2.pdf}\end{minipage}  
&= v_i A^{-\frac{1}{2}} (A^{-2} -1)^{-1}
 \left(  \begin{minipage}{.5in}\includegraphics[width=\textwidth]{rel-punctureskein1.pdf}\end{minipage} 
-  A^{-1}\begin{minipage}{.5in}\includegraphics[width=\textwidth]{rel-punctureskein1a.pdf}\end{minipage} 
 \right).
\end{align*}
So when $A^2 -1$ is invertible, if a set of elements generates  only the simple arcs, then it generates all of $\mathcal A(F_{g,n})$ by Lemma 3.1.  
However, in the following examples, we will work under the most general set-up, and we will \emph{not} assume that $A^2 -1$ is invertible.

\subsection{Arc Algebra of Punctured Spheres}

We begin by a refinement of Lemma~\ref{lem:simple} in the case of punctured spheres, $F_{0, n}$ with $n \geq 2$.  

\begin{prop} \label{lem:arcs}
If a set of elements $\{ x_1, x_2, \ldots, x_n \}$ of the arc algebra $\mathcal A(F_{0,n})$ generate the simple arcs, then it generates the entire algebra. 
\end{prop}

\begin{proof}
By Lemma \ref{lem:simple}, it suffices to show that any simple knot can be rewritten in terms of simple arcs.  
Given a simple knot $K \subseteq F_{0, n} \times[0,1]$, notice that it has a projection which separates $F_{0, n}$ into two punctured disks.  Let $D$ be the punctured disk bounded by $K$ with the smaller number of punctures, say $p_1, \ldots, p_k$.   

Since $K$ is a simple knot, $k \geq 2$.  There exist two disjoint simple arcs from $p_1$ to $p_k$ such that the union of their projections onto $F_{0, n}$ encloses the remaining punctures $p_2, \ldots, p_{k-1}$.   Let $\alpha$ and $\beta$ be the skeins represented by these two arcs, respectively.   See Figure \ref{fig:lemarcs}.   
\begin{figure}
\includegraphics[width=.20\textwidth]{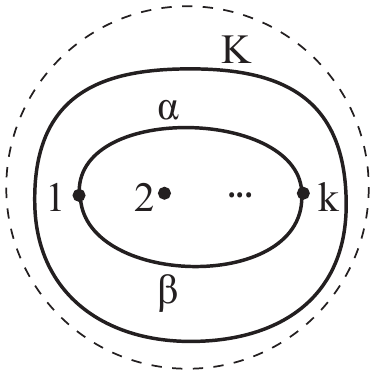}
\caption{A neighborhood of $K$ in $F_{0, n}$.}
\label{fig:lemarcs}
\end{figure}

Consider the product $\alpha \ast \beta \in \mathcal A( F_{0,n} )$ and apply the puncture-skein relation twice:
\begin{eqnarray*}
\alpha \ast \beta  &=& \pic{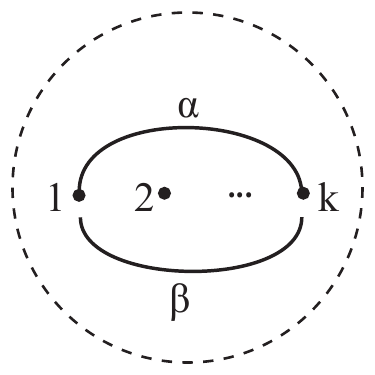} \\
 &=& \frac{1}{v_1v_k}\left(A\pic{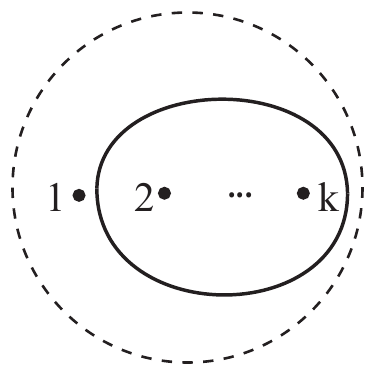} + \pic{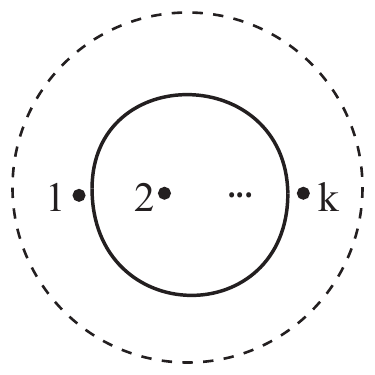} + \pic{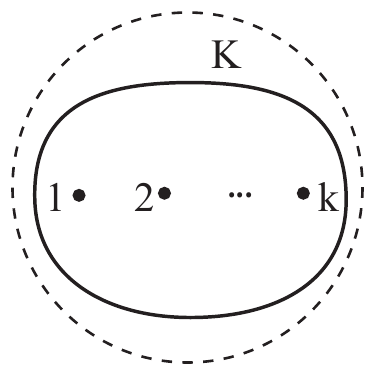} + A^{-1}\pic{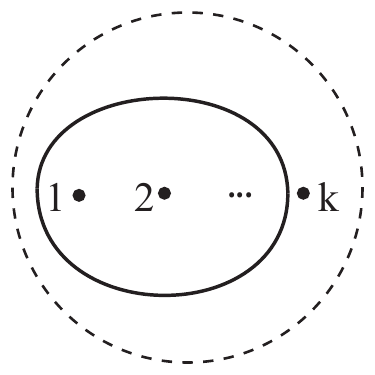}\right).
\end{eqnarray*}
Thus $[K]$ can be rewritten as a linear combination involving the product of two simple arcs ($\alpha$ and $\beta$) and three knots bounding disks with strictly fewer punctures.   Notice also that a knot bounding a disk with one or no punctures can be removed using the puncture relation or the framing relation, respectively.  Thus by induction, we are done.  
\end{proof}

\subsubsection{Sphere with 2 Punctures}


\begin{thm} \label{thm:a02}
$\mathcal A(F_{0,2}) = R_2\langle \alpha \; | \;  \alpha^2= -v_1^{-1} v_2^{-1}   (A-A^{-1})^2 \rangle  $ where $\alpha$ is represented by a simple arc between the two punctures of $F_{0, 2}$.
\end{thm}

\begin{proof}
On $F_{0, 2}$, any simple arc must start at one puncture and end at the other without intersecting itself.   Up to isotopy, there is only one such arc, and let $\alpha$ be the skein represented by that arc. By Proposition~\ref{lem:arcs}, $\alpha$ generates the algebra.
Note that in the  arc algebra, 
\begin{align*}
\alpha^2 &= \pic{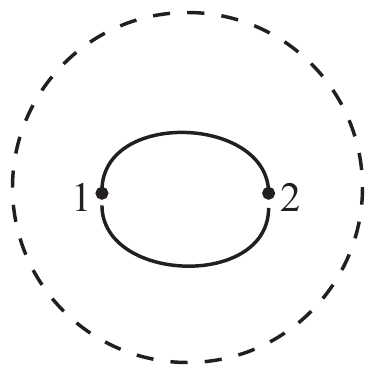} = \frac{1}{v_1 v_2} \left(A\pic{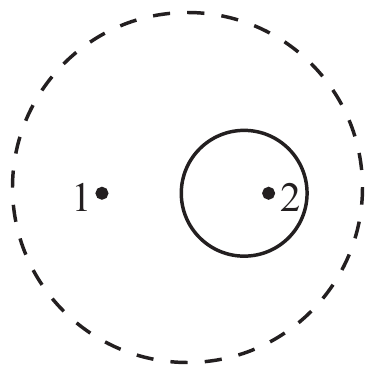} + 2\pic{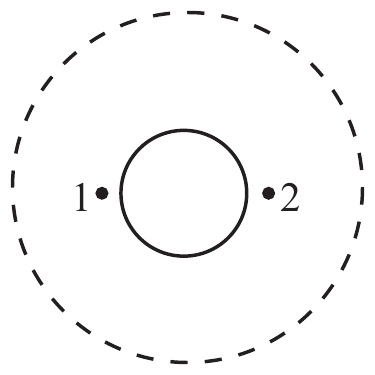} + A^{-1}\pic{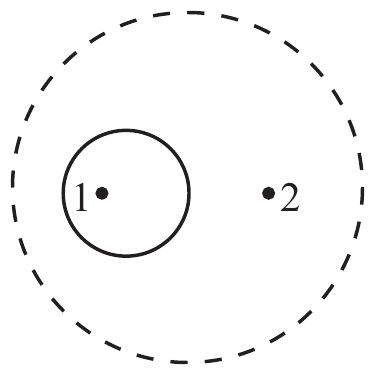}\right) \\
&= v_1^{-1} v_2^{-1} ((A+A^{-1})(A+A^{-1}) + 2(-A^2-A^{-2})) \\
&= -v_1^{-1} v_2^{-1} (A - A^{-1})^2.
\end{align*}
In particular, this shows that $\alpha^2$ is not linearly independent from $1$ and $\alpha$.  As $\alpha$ is the only generator, this is the only relation of $\mathcal A(F_{0,2})$.
\end{proof}

\subsubsection{Sphere with 3 Punctures}
To determine the  arc algebra $\mathcal A(F_{0, 3})$, we will need the following lemma from algebra.  
\begin{lem} \label{lem:linind}
Let $A$ and $B$ be $R$-algebras.
Suppose $x_1, x_2, \dots, x_n$ are elements of the algebra $A$ and $\rho$ is some algebra homomorphism of $A$ to the algebra $B$.
If the elements $\rho(x_1), \rho(x_2), \dots, \rho(x_n)$ are linearly independent in $B$, then they are linearly independent in $A$.
\end{lem}
\begin{proof}
We prove the contrapositive.
Suppose $x_1, x_2, \dots, x_n$ are linearly dependent in $A$.
Then there exist coefficients $k_1, k_2, \dots, k_n \in R$ so that
$
k_1x_1 + k_2x_2 + \dots + k_nx_n = 0
$
and at least one $k_i$ is nonzero.
Since $\rho$ is an $R$-algebra homomophism, 
$
k_1\rho(x_1) + k_2\rho(x_2) + \dots + k_n\rho(x_n) = 0.
$
So $\rho(x_1), \rho(x_2), \dots, \rho(x_n)$ are linearly dependent in $B$ as well.
\\
\end{proof}

\begin{thm} \label{thm:a03}
 $\mathcal A( F_{0,3} ) = R_3  \langle \alpha_1, \alpha_2, \alpha_3 \; |\;  \alpha_i \alpha_{i+1} = \alpha_{i+1} \alpha_i = v_{i+2}^{-1} \delta \, \alpha_{i+2},\,\, v_{i+1} v_{i+2} \alpha_i^2 = \delta^2 \rangle,  $
where $\alpha_i$ is represented by the simple arc connecting the punctures $p_{i+1}$ to $p_{i+2}$ in the thrice-punctured sphere $F_{0,3}$, with $i = 1,2,3$ and indices interpreted modulo 3, and where $\delta = (A^\frac{1}{2} + A^{-\frac{1}{2}})$.
\end{thm}

\begin{proof}
Since the only simple arcs in $F_{0,3}$ are those connecting distinct punctures, it follows that  $\alpha_1, \alpha_2$, and $\alpha_3$ generate the arc algebra $K[F_{0, 3}]$.  
Observe that
\begin{align*}
\alpha_i ^2 &= \pic{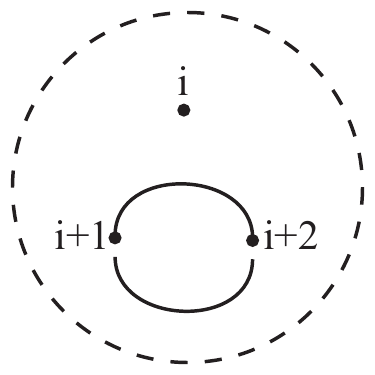}\\
&= v_{i+1}^{-1} v_{i+2}^{-1}\left(A\pic{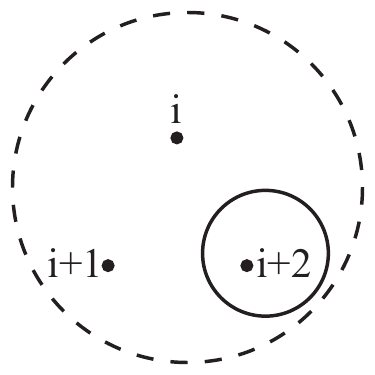} + \pic{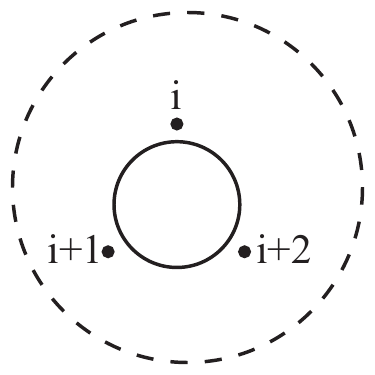} + \pic{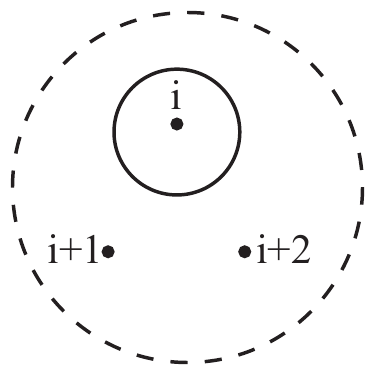} + A^{-1}\pic{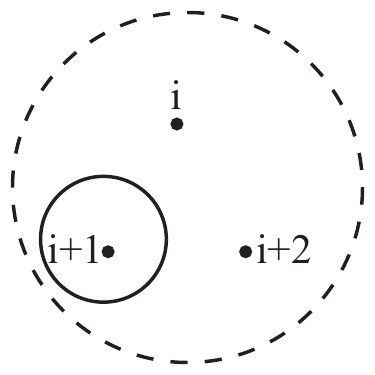}\right) \\
&=v_{i+1}^{-1} v_{i+2}^{-1}(A(A+A^{-1}) + (-A^2-A^{-2}) + (A+A^{-1}) + A^{-1}(A+A^{-1})) \\
&= v_{i+1}^{-1} v_{i+2}^{-1}(A^\frac{1}{2} + A^{-\frac{1}{2}})^2
\\
&= v_{i+1}^{-1} v_{i+2}^{-1}\delta^2.
\end{align*}
Also,
\begin{align*}
\alpha_i \ast \alpha_{i+1}  &= \pic{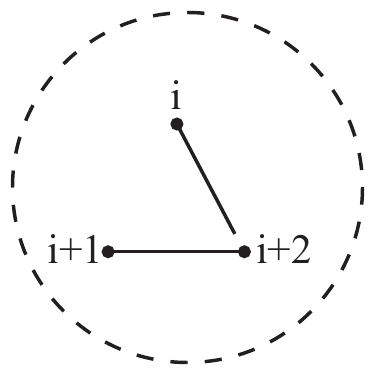} = v_{i+2}^{-1}\left(A^\frac{1}{2}\pic{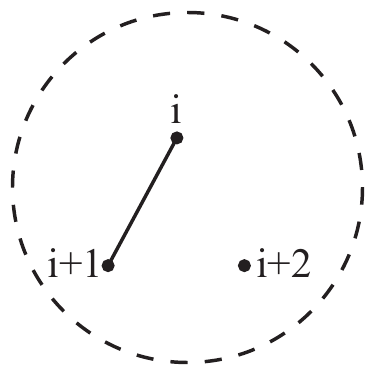} + A^{-\frac{1}{2}}\pic{thma03-7.pdf}\right) \\
&= v_{i+2}^{-1}(A^\frac{1}{2} + A^{-\frac{1}{2}})\alpha_{i+2} 
\\
&= v_{i+2}^{-1}\delta\alpha_{i+2}
\end{align*}
and similarly
\begin{align*}
\alpha_{i+1} \ast \alpha_i&= \pic{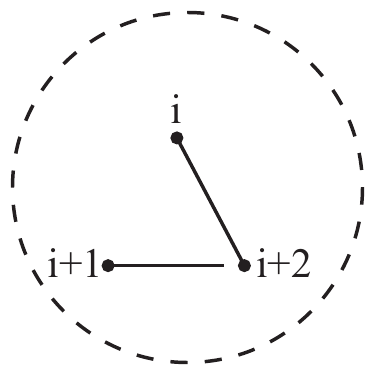} = v_{i+2}^{-1}\left(A^\frac{1}{2}\pic{thma03-7.pdf} + A^{-\frac{1}{2}}\pic{thma03-7.pdf}\right) \\
&= v_{i+2}^{-1}(A^\frac{1}{2} + A^{-\frac{1}{2}})\alpha_{i+2}
\\
&= v_{i+2}^{-1}\delta\alpha_{i+2}.
\end{align*}

We next show that these are the only relations.  Notice that the relations above imply that any product $\alpha_i \ast \alpha_j $ can be rewritten as either a scalar multiple when $i=j$ or as a multiple of the remaining $\alpha_k$, for $k \neq i, j$.  
Thus any word in $\alpha_1$, $\alpha_2$, and $\alpha_3$ can be rewritten in terms of a scalar multiple of one or zero generators.
So any other relation among the generators $\alpha_1$, $\alpha_2$, and $\alpha_3$ can be expressed in the form
\[
k_0 + k_1\alpha_1 +k_2\alpha_2 + k_3\alpha_3 = 0
\]
where $k_i \in R_3$.   We will show that $1, \alpha_1, \alpha_2, \alpha_3$ are linearly independent, so that the $k_i = 0$.

Recall that a left regular representation of a group is the linear representation provided by multiplication of group elements on the left.  Based on the similarity of the  algebra elements $1, \alpha_1, \alpha_2, \alpha_3$ from $\mathcal A( F_{0,3} ) $ with the group elements of $\mathbb Z_2 \times \mathbb Z_2$, we  define a left regular representation $\rho$  for $\mathcal A( F_{0,3} ) $ by: 
\[
\rho(1) =
\begin{bmatrix}
1 & 0 & 0 & 0 \\
0 & 1 & 0 & 0 \\
0 & 0 & 1 & 0 \\
0 & 0 & 0 & 1
\end{bmatrix},
\quad
\rho(\alpha_1) =
\begin{bmatrix} \footnotesize
0 & v_2^{-1} v_3^{-1}\delta^2 & 0 & 0 \\
1 & 0 & 0 & 0 \\
0 & 0 & 0 & v_2^{-1}\delta \\
0 & 0 & v_3^{-1}\delta & 0
\end{bmatrix},
\]
\[
\rho(\alpha_2) =
\begin{bmatrix}
0 & 0 & v_1^{-1}v_3^{-1}\delta^2 & 0 \\
0 & 0 & 0 & v_1^{-1}\delta \\
1 & 0 & 0 & 0 \\
0 & v_3^{-1}\delta & 0 & 0
\end{bmatrix},
\quad
\rho(\alpha_3) =
\begin{bmatrix}
0 & 0 & 0 & v_1^{-1}v_2^{-1}\delta^2 \\
0 & 0 & v_1^{-1}\delta & 0 \\
0 & v_2^{-1}\delta & 0 & 0 \\
1 & 0 & 0 & 0
\end{bmatrix}.
\]
Note that the coefficients from each column are exactly those given by the equations describing left multiplication by $\alpha_i$.  In particular, they are the coefficients in the equations $\alpha_i \ast 1 = \alpha_i$,  $\alpha_i \ast \alpha_i = v_{i+1}^{-1} v_{i+2}^{-1}\delta^2 $, and $\alpha_i \ast  \alpha_j  = v_{i+2}^{-1}\delta \; \alpha_{i+2}$ for $i \neq j$.   
The  matrices $\rho(1)$, $\rho(\alpha_1)$, $\rho(\alpha_2)$ and $\rho(\alpha_3)$ are clearly linearly independent, as can be determined by looking at their first columns.
Thus by Lemma~\ref{lem:linind} the original $1, \alpha_1, \alpha_2, \alpha_3$ are linearly independent.
Hence there are no more relations to be found in $K( F_{0,3} )$.  
This also shows that $\{ \alpha_1, \alpha_2, \alpha_3 \}$ is a minimal set of generators.
\end{proof}

\subsection{Surfaces with $0$ or $1$ Punctures.}  \label{sec:n01}

Recall that $R=\Z[A, A^{-1}]$ is a subring of $R_n = \Z[A^{\frac{1}{2}}, A^{-\frac{1}{2}}][v_1^{\pm1}, v_2^{\pm1}, \dots, v_n^{\pm1}]$, and that  from Lemma~\ref{lem:psi}, there exists an algebra homomorphism $\psi$, which maps the $R$-algebra  $\mathcal S(F_{g,n})$ to the $R_n$-algebra $\mathcal A(F_{g,n})$.   

First observe that when $n=0$, the relations in $\mathcal S(F_{g,0})$ are exactly those in $\mathcal A(F_{g,0})$.  That is, $R_0 \otimes \mathcal K_0(F_{g,0}) \cong \mathcal K(F_{g,0})$.  Moreover, the map $\psi$ from the proof of Lemma~\ref{lem:psi} is injective when $n = 0$ and acts as the identity on simple knots.  
Since all simple curves are simple knots in this case and the image of $\psi$ contains all simple knots, the image of $\psi$ generates all of the arc algebra $\mathcal A(F_{g,0})$  by Lemma \ref{lem:simple}.  
Thus  $R_0 \otimes \mathcal S( F_{g,0}) \cong \mathcal A(F_{g,0})$, and any presentation of $\mathcal S( F_{g,0})$ provides a presentation of $\mathcal A(F_{g,0})$.

When $n=1$, again there are no simple arcs, so that the image of $\psi$ generates all of the arc algebra $\mathcal A(F_{g,1})$.  So any set generating $\mathcal S(F_{g,1})$ also generates $\mathcal A(F_{g,1})$.     However, the map $\psi$ is no longer injective.   Specifically, the relations for the Kauffman skein algebra and the relations for the Kauffman arc algebra will differ; the puncture-framing relation from the Kauffman arc algebra is not a relation in the Kauffman skein algebra.   However, this is the only difference.   
Hence $R_1 \otimes \mathcal S( F_{g,1}) / \mathcal K_\mathrm{pfr} (F_{g,1}) \cong \mathcal A(F_{g,1})$, where  $K_\mathrm{pfr}(F_{g,1})$ is the submodule generated by only the puncture-framing relation.  In summary,  the generators of $\mathcal A(F_{g,1})$ are generators of $\mathcal S(F_{g,1})$, but the relations of $\mathcal A(F_{g,1})$ are relations of $\mathcal S(F_{g,1})$ along with one corresponding to the puncture-framing relation.

\subsubsection{Torus with 0 or 1 punctures} 

As an example, let us examine the case of the closed torus and the torus with one puncture.  From \cite{BuPr00}, we have that the Kauffman skein algebras $\mathcal S(F_{1, 0})$ and $\mathcal S(F_{1,1})$ are both generated as a $\Z[A, A^{-1}]$-module by three simple closed curves $\gamma_1, \gamma_2, \gamma_3$ such that $\gamma_1$ and $\gamma_2$ intersect once and $\gamma_3$ is one of two curves that meet both $\gamma_1 $ and $\gamma_2$ once.  Moreover, if $\partial$ represents a small loop around the puncture of $F_{1,1}$, then 
\begin{equation} \label{eqn:boundary}
\partial =  A \gamma_1 \gamma_2 \gamma_3 - A^2 \gamma_1^2  - A^{-2} \gamma_2^2 -  A^2 \gamma_3^2 + A^2 + A^{-2}.
\end{equation}
In the skein algebra $\mathcal S(F_{1,0})$, we have $\partial = - A^2 - A^{-2}$. 
Up to a change in scalars from $R$ to $R_0$,  a presentation of the arc algebra $\mathcal A(F_{1,0})$ is the same as the presentation of the skein algebra $\mathcal S(F_{1,0})$.  That is,
\begin{eqnarray*}
 \mathcal A(F_{1,0}) = R_0 \; \langle \gamma_1, \gamma_2, \gamma_3  &|&
    A \gamma_i \gamma_{i+1} - A^{-1} \gamma_{i+1} \gamma_i = (A^2 - A^{-2}) \gamma_{i+1} 
    \mbox{ \emph{ and } }\\
\;  
&& -A^2 -  A^{-2} = A\gamma_1 \gamma_2 \gamma_3 - A^2 \gamma_1^2 - A^{-2} \gamma_2^2 -A^2 \gamma_3^2 + A^2 + A^{-2} \rangle
\end{eqnarray*}
where the indices are interpreted modulo 3.
On the other hand,  in the once-punctured torus, we have $\partial = A + A^{-1}$ in the arc algebra.  Thus 
\begin{eqnarray*}
\mathcal A(F_{1,1}) = R_1 \langle \gamma_1, \gamma_2, \gamma_3 \; 
&|& 
\;  A \gamma_i \gamma_{i+1} - A^{-1} \gamma_{i+1} \gamma_i = (A^2 - A^{-2}) \gamma_{i+1} 
\mbox{ \emph{ and } }\\
\;  
&& A+  A^{-1} = A\gamma_1 \gamma_2 \gamma_3 - A^2 \gamma_1^2 - A^{-2} \gamma_2^2 -A^2 \gamma_3^2 + A^2 + A^{-2} \rangle.
\end{eqnarray*}
\\

\noindent {\bf Acknowledgements}.  The authors would like to thank Eric Egge and Francis Bonahon for helpful discussions, and the Carleton College mathematics department for their support and encouragement throughout this research.  

\bibliographystyle{plain}
\nocite{*}

\end{document}